\documentclass[12pt]{amsart}

\usepackage{cancel}
\usepackage{latexsym, amssymb, amsmath,esint}
\usepackage{soul}
\usepackage{amsfonts, graphicx}
\usepackage{graphicx,color}

\usepackage{hyperref}
\usepackage[alphabetic]{amsrefs}

\newcommand{\be}{\begin{equation}}
\newcommand{\ee}{\end{equation}}
\newcommand{\beq}{\begin{eqnarray}}
\newcommand{\eeq}{\end{eqnarray}}

\usepackage{wasysym,stmaryrd}
\newtheorem{thm}{Theorem}[section]
\newtheorem{conj}{Conjecture}[section]

\newtheorem{lma}[thm]{Lemma}
\newtheorem{prop}[thm]{Proposition}
\newtheorem{cor}[thm]{Corollary}

\theoremstyle{remark}
\newtheorem{rem}[thm]{Remark}
\numberwithin{equation}{section}

\newtheorem{claim}{Claim}[section]

\def\be{\begin{equation}}
\def\ee{\end{equation}}
\def\bee{\begin{equation*}}
\def\eee{\end{equation*}}

\def\Ric{\text{\rm Ric}}
\def\Rm{\text{\rm Rm}}

\def\tr{\operatorname{tr}}

\def\e{\varepsilon}

\begin{document}

\title[]
{Quantification of scalar curvature under $C^0$ convergence using smoothing}

 \author{Man-Chun Lee}
\address[Man-Chun Lee]{Department of Mathematics, The Chinese University of Hong Kong, Shatin, Hong Kong, China}
\email{mclee@math.cuhk.edu.hk}


\date{\today}

\begin{abstract}
A quantitative version of the scalar lower bound under $C^0$ convergence was conjectured by Gromov. More recently, Mazurowski and Yao proved that a refined form of Gromov’s conjecture holds in dimension three. Furthermore, they constructed examples demonstrating that such a refinement is necessary. In this paper, we establish that the refined quantitative bound holds in all dimensions greater than or equal to three.
\end{abstract}

\maketitle

\markboth{ Man-Chun Lee}{Quantification using Ricci flow}

\section{Introduction}\label{sec: introduction}
In studying the compactness properties of scalar curvature, Gromov proposed investigating whether lower bounds on scalar curvature are preserved under weak convergence of Riemannian metrics. In \cite{Gromov2014}, he showed that on a fixed (possibly noncompact) manifold, a scalar curvature lower bound is preserved under $C^0$ convergence of metrics. Gromov’s argument relies on reformulating positive scalar curvature in terms of a $C^0$-stable quantity.
An alternative proof based on Ricci flow was later developed by Bamler \cite{Bamler2016}, building on the stability results for Ricci flow established by Koch and Lamm \cite{KochLamm2012}. The Ricci flow approach was subsequently used by Burkhardt-Guim \cite{Burkhardt2019} to introduce a notion of scalar curvature lower bounds for $C^0$ metrics.
A natural and longstanding question is whether scalar curvature lower bounds remain preserved under convergence weaker than $C^0$; see, for example, the discussions in \cite{Gromov4Lecture,SormaniSurvey}. We also refer to \cite{ChuLee2025,LeeTopping2022,KazarasXu} for related results in this direction.

In \cite[Section 3.1.3]{Gromov4Lecture}, Gromov raised the question of whether the $C^0$ convergence theorem admits a quantitative version:
\begin{conj}\label{conj:Gromo}
Let $g_\e,g_0$ be  smooth metrics on the Euclidean ball $B_{euc}(1)\Subset \mathbb{R}^n$, then there exists $C(g_0),\e_0(g_0)$ such that 
\begin{equation}\label{eqn:conj-eq}
\inf_{B_{euc}(1)}\mathrm{scal}_{g_\e} \leq \mathrm{scal}_{g_0}(0)+C ||g_0-g_\e||_{L^\infty(B_{euc}(1),g_0)}
\end{equation}
if $||g_0-g_\e||_{L^\infty(B_{euc}(1),g_0)}=\e\leq \e_0$.
\end{conj}

The conjecture was recently studied by Mazurowski and Yao \cite{MazurowskiYao2026}. They constructed examples showing that Gromov’s original quantitative formulation of $C^0$ convergence cannot hold in dimensions $n\geq 3$. Their counterexample is rotationally symmetric and conformal to Euclidean space, see Appendix~\ref{sec:appe}. Motivated by this construction, they proposed a refined version of Conjecture~\ref{conj:Gromo}, in which the term $||g_0-g_\e||_{L^\infty(B_{euc}(1),g_0)}$  in \eqref{eqn:conj-eq} is replaced by  $||g_0-g_\e||_{L^\infty(B_{euc}(1),g_0)}^{1/2}$. Among other results, they proved this refined conjecture in the case $n=3$.

\begin{thm}[Theorem 2 in \cite{MazurowskiYao2026}]\label{thm:MazurowskiYao}
Let $g_\e,g_0$ be  smooth metrics on the Euclidean ball $B_{euc}(1)\Subset \mathbb{R}^3$, then there exists $C(g_0),\e_0(g_0)$ such that 
\begin{equation}\label{eqn:conj-eq}
\inf_{B_{euc}(1)}\mathrm{scal}_{g_\e} \leq \mathrm{scal}_{g_0}(0)+C ||g_0-g_\e||^{1/2}_{L^\infty(B_{euc}(1),g_0)}
\end{equation}
if $||g_0-g_\e||_{L^\infty(B_{euc}(1),g_0)}:=\e\leq \e_0$.
\end{thm}

Their approach relies on the stability theory for elliptic PDEs together with a novel application of the monotonicity formula for harmonic functions in dimension three, due to Agostiniani, Mazzieri, and Oronzio \cite{AgostinianiMazzieriOronzio2024}. This monotonicity formula ultimately exploits the Gauss–Bonnet theorem on level set and is therefore inherently restricted to $n=3$. Motivated by the work of Mazurowski and Yao \cite{MazurowskiYao2026}, we consider the problem of quantifying $C^0$ convergence in dimensions $n \geq 3$, using smoothing method. Our main result extends Theorem~\ref{thm:MazurowskiYao} to all dimensions $n \geq 3$.

\begin{thm}\label{thm:main}
For $n\geq 3$, there exists $\e_n>0$ such that the following is true. Suppose $(M^n,g_0)$ is a Riemannian manifold and $\Omega:=B_{g_0}(x_0,r)\Subset M$ for some $r>0$ such that 
\begin{enumerate}
\item[(i)] $\Omega\Subset M$;
\item[(ii)] $\sum_{k=0}^2 r^{k}\sup_{\Omega} |\nabla^{g_0,k}\Rm(g_0)|\leq r^{-2}$;
\end{enumerate}
then there is  $\Lambda(n)>0$ such that
\begin{equation*}
\inf_{\Omega}\mathrm{scal}_{\hat g_0} \leq \mathrm{scal}_{g_0}(x_0)+\Lambda r^{-2} ||\hat g_0-g_0||_{L^\infty(\Omega,g_0)}^\frac{1}{2}
\end{equation*}
for any smooth metric $\hat g_0$ on $\Omega$ with $||\hat g_0-g_0||_{L^\infty(\Omega,g_0)}<\e_n$.
\end{thm}

We formulate the quantitative result on intrinsic  domains in order to emphasize the precise dependence on the geometry of $g_0$. In particular, the quantification does not rely on any non-collapsing assumption for $g_0$ and is therefore applicable even in collapsing settings. We also address the case in which the two metrics are close in the local $L^p$ sense under an additional non-collapsing assumption.

\begin{thm}\label{thm:Lp-quant}
For $n\geq 3$, there exists $\e_n>0$ such that the following is true. Suppose $(M^n,g_0)$ is a Riemannian manifold and $\Omega:=B_{g_0}(x_0,r)\Subset M$ for some $r>0$ such that 
\begin{enumerate}
\item[(i)] $\Omega\Subset M$;
\item[(ii)] $\sum_{k=0}^2 r^{k}\sup_{\Omega} |\nabla^{g_0,k}\Rm(g_0)|\leq r^{-2}$;
\item[(iii)] $\mathrm{Vol}_{g_0}(\Omega)\geq v_0r^n$, for some $v_0>0$,
\end{enumerate}
then there is  $L(n,v_0)>0$ such that
\begin{equation*}
\inf_{\Omega}\mathrm{scal}_{\hat g_0} \leq \mathrm{scal}_{g_0}(x_0)+L r^{-2-\frac{2n}{4p+n}} ||\hat g_0-g_0||_{L^p(\Omega,g_0)}^\frac{1}{2+\frac{n}{2p}}
\end{equation*}
for any smooth metric $\hat g_0$ on $\Omega$ with $||\hat g_0-g_0||_{L^\infty(\Omega,g_0)}<\e_n$.
\end{thm}

The exponent in Theorem~\ref{thm:Lp-quant} is sharp, as demonstrated by the example constructed in \cite{MazurowskiYao2026}; see Appendix~\ref{sec:appe}. This $L^p$ quantitative result is partly motivated by the examples in \cite{LeeTopping2022}. As an application, we extend \cite[Theorem~6]{MazurowskiYao2026} to all dimensions $n\geq 3$:
\begin{cor}
There exists $\e_n>0$ such that if $g_0$ is a smooth metric on $M$ (possibly open) and $g_i$ is a sequence of smooth metric on $M$ such that 
\begin{enumerate}
\item[(i)] $g_i\to g_0$ in sense of measure locally;
\item[(ii)] $||g_i-g_0||_{L^\infty(M,g_0)}<\e_n$;
\item[(iii)] $\mathrm{scal}(g_i)\geq \kappa$ on $M$ for some $\kappa\in \mathbb{R}$,
\end{enumerate}
then $\mathrm{scal}(g_0)\geq \kappa$ on $M$.
\end{cor}

Our approach is based on the stability of Ricci flow smoothing. One of the key ingredients is closeness in the $L^\infty$ norm. Informally, this condition ensures that the tangent cone of $\hat g_0$ in Theorem~\ref{thm:main} or Theorem~\ref{thm:Lp-quant}, taken with respect to a background metric $g_0$ of controlled geometry, is Euclidean. This allows us to compare their Ricci flow smoothing. More precisely, one may think informally of an estimate of the form
$$||\hat g(t)-g(t)||_{C^k_{loc}}=o(||\hat g_0-g_0||_{L^\infty} )t^{-k/2}.$$
Quantitatively, however, this convergence is only scaling invariant in time.

On the other hand, it is standard that a scalar curvature lower bound is preserved under the Ricci flow, at least in the compact case. From a quantitative perspective, the preservation of scalar lower bound break the scaling invariance. In the context of proving Gromov’s $C^0$-convergence theorem \cite{Gromov2014}, this suffices, since it implies that the limit of $\hat g(t)$ coincides with the smooth Ricci flow $g(t)$. Consequently, the scalar curvature lower bound for $\hat g(t)$ passes directly to $g(t)$. By letting $t\to0$, this recovers the scalar lower bound for $g(0)=g_0$. This strategy was employed by Bamler \cite{Bamler2016}. 

In this paper, we refine and extend the methodology of \cite{Bamler2016}. By exploiting the regularity of $g_0$, we compare $\mathrm{scal}_{g_0}$ directly with $\mathrm{scal}_{g(t)}$, thereby breaking the scaling invariance. The desired result then follows by choosing the time scale $t$ appropriately. One of the main challenges is localization, which relies on local analysis developed by the author in other settings.

Organization of the paper. Section~\ref{sec:stability} contains preliminary results on the stability of the Ricci–DeTurck flow on Euclidean space. In Section~\ref{sec:local-esti}, we derive a priori local estimates for the Ricci–DeTurck flow. Section~\ref{sec:proof} is devoted to the proofs of the quantification results for both $C^0$ and $L^p$ convergence. Finally, in Section~\ref{sec:appe}, we discuss the optimality of the exponent appearing in Theorem~\ref{thm:Lp-quant}.

\subsection*{Acknowledgment}   The author was partially supported by Hong Kong RGC grants No. 14300623 and No. 14304225, and an Asian Young Scientist Fellowship.

\section{Stability of Ricci flow}\label{sec:stability}

In this section, we collect some useful facts concerning Ricci flow and its stability. The Ricci flow is a one parameter family of metric $\tilde g(t)$ which solves 
\begin{equation}\label{eqn:RF-equ}
\left\{
\begin{array}{ll}
\partial_t \tilde g(t)=-2\Ric(\tilde g(t)); \\
\tilde g(0)=\tilde g_0.
\end{array}
\right.
\end{equation}

To quantify scalar curvature lower bounds under $C^0$ perturbation, we make use of Ricci flow stability theory.For our purposes, it suffices to consider the stability of the Euclidean metric, that is, the trivial Ricci flow on $\mathbb{R}^n$. Given a Ricci flow $\tilde g(t)$ on $N\times [0,T]$, the Ricci-DeTurck flow with respect to $\tilde g(t)$ is a one parameter family of metrics $g(t)$ satisfying 
\begin{equation}\label{eqn:RDF-equ}
\left\{
\begin{array}{ll}
\partial_t  g_{ij}=-2R_{ij}+\nabla_i W_j+\nabla_j W_i; \\
W^k=g^{ij}\left(\Gamma_{ij}^k-\tilde\Gamma_{ij}^k \right)
\end{array}
\right.
\end{equation}
where $\Gamma$ and $\tilde\Gamma$ denote connection of $g(t)$ and $\tilde g(t)$ respectively. 

Unlike the Ricci flow equation, the Ricci-DeTurck flow is a strictly parabolic system and is diffeomorphic to a Ricci flow $G(t):=\Psi_t^*g(t)$ with $G(0)=g(0)$ if $\Psi_t$ is a one parameter family of diffeomorphism such that 
\begin{equation}\label{eqn:RD-ODE}
\left\{
\begin{array}{ll}
\partial_t \Psi_t(x)=-W\left( \Psi_t(x),t\right);\\[1mm]
\Psi_0=\mathrm{Id}.
\end{array}
\right.
\end{equation} 

We will require the following (weak) stability result for Ricci flow due to Koch and Lamm \cite{KochLamm2012} in the case where $\tilde g(t)\equiv g_{euc}$ on $\mathbb{R}^n$, see also \cite{Burkhardt2019,CaiWang2026,ChanLaiLee2025,Simon}. We use the version from \cite{ChanLaiLee2025} which avoid the introduction of new banach spaces.

\begin{thm}\label{thm:stability-RDF}
There exists dimensional constant $\e_0,C_0>0$ such that if $g_0$ is a $L^\infty$ metric on $\mathbb{R}^n$ such that $||g_0-g_{euc}||_{L^\infty}\leq \e_0$, then there is a solution $g(t)$ to the Ricci-DeTurck flow (with respect to $g_{euc}$) on $\mathbb{R}^n\times (0,+\infty)$ such that  
\begin{equation}\label{eqn:closeness}
||g(t)-g_{euc}||_{L^\infty(g_{euc})}\leq C_0 ||g_0-g_{euc}||_{L^\infty(g_{euc})}.
\end{equation}
Furthermore if $g(t)$ and $\hat g(t)$ are both solution to the Ricci-DeTurck flow such that \eqref{eqn:closeness} holds, then 
\begin{equation}
\sum_{k=0}^2 t^{k/2}||\nabla^{euc,k}\left(g(t)-\hat g(t)\right)||_{L^\infty(g_{euc})}\leq C_0 ||g_0-\hat g_0||_{L^\infty(g_{euc})}.
\end{equation}
\end{thm}
\begin{proof}
This is special case of \cite[Theorem 1.3]{ChanLaiLee2025} where the reference Ricci flow $\tilde g(t)$ is now static Euclidean space. See also \cite[Theorem 4.3]{KochLamm2012}.
\end{proof}

The following proposition gives a improved control of $g(t)-\hat g(t)$, when the initial data $|g_0-\hat g_0|$ enjoys better control in an integral sense. This will play a crucial role in Theorem~\ref{thm:Lp-quant}.
\begin{prop}\label{prop:improve-L-infty-Lp}
Suppose $g(t),\hat g(t)$ are Ricci-DeTurck flow (with respect to $g_{euc}$ on $\mathbb{R}^n\times [0,1]$, constructed in Theorem~\ref{thm:stability-RDF}, such that \eqref{eqn:closeness} hold, then there exists $C_2(n)>0$ such that for any $p\geq 1$, we have 
\begin{equation}
||g(t)-\hat g(t)||_{L^\infty(g_{euc})}\leq C_2 t^{-\frac{n}{2p}}||g_0-\hat g_0||_{L^p(g_{euc})}
\end{equation}
for all $t\in (0,1]$.
\end{prop}
\begin{proof}
This follows implicitly from the proof of \cite[Theorem 4.10]{ChanLaiLee2025}, as well as that of \cite[Theorem 4.3]{KochLamm2012} using the heat kernel representation formula. By \cite[(4.10)]{ChanLaiLee2025}, Euclidean heat kernel estimate and shrinking $\e_0$ if necessary, we have 
\begin{equation}
||g(t)-\hat g(t)||_{L^\infty(g_{euc})}\leq 2 \cdot \sup_{x\in M}\int_{\mathbb{R}^n}  K(x,y,t) \cdot  |g_0-\hat g_0|(y)\, d\mathrm{vol}_{euc}(y)
\end{equation}
where $K(x,y,t)$ denotes the heat kernel on $\mathbb{R}^n$. Here we have used the Euclidean space so that conjugate heat kernel coincides with the standard heat kernel.

In particular the right hand side can be controlled as follows. If $p>1$, then 
\begin{equation}
\begin{split}
&\quad \int_{\mathbb{R}^n}  K(x,y,t) \cdot  |g_0-\hat g_0|(y)\, d\mathrm{vol}_{euc}(y)\\
&\leq   \left(\int_{\mathbb{R}^n}  K(x,y,t)^q \,dy \right)^\frac1q \cdot ||g_0-\hat g_0||_{L^p}\\
&\leq  C_2 t^{-\frac{n}{2p}} ||g_0-\hat g_0||_{L^p}\\
\end{split}
\end{equation}
where $q$ is the conjugate of $p$. The same also holds if $p=1$.
\end{proof}

\section{Local estimates of Ricci-DeTurck flow}\label{sec:local-esti}

In this section, we establish local estimates for the Ricci–DeTurck flow with respect to $g_{euc}$.  We begin with a local persistence result for scalar curvature lower bounds, which is based on the local maximum principle developed by the author and Tam \cite{LeeTam2022}, see also Bamler’s approach in  \cite{Bamler2016}.

\begin{lma}\label{lma:scalar-low-bdd}
Suppose $g(t)$ is a Ricci-DeTurck flow (with respect to $g_{euc}$) on $\mathbb{R}^n\times [0,1]$ with $g(0)=g_0$ such that 
\begin{enumerate}
\item[(i)] $\displaystyle 2^{-1} g_{euc}\leq g(t)\leq 2 g_{euc}$; 
\item[(ii)] $|\nabla^{euc,2}g(t)|+|\nabla^{euc} g(t)|^2\leq t^{-1}$
\end{enumerate}
on $\mathbb{R}^n\times (0,1]$. If in addition $\mathrm{scal}(g_0)\geq \kappa$ on $B_{g_0}(x_0,4)$, then  for all $\ell\in \mathbb{N}$, there exists $T_\ell(n)\in (0,1]$ such that 
\begin{equation}
\mathrm{scal}(x_0,t)\geq \kappa -t^\ell
\end{equation}
for all $t\in [0,T_\ell]$.
\end{lma}
\begin{proof}
It is more convenient to work on the Ricci flow. To make it precise, we consider the Ricci-DeTurck ODE in \eqref{eqn:RD-ODE}: 
\begin{equation}
\left\{
\begin{array}{ll}
\partial_t \Psi_t(x)=-W\left(\Psi_t(x),t \right);\\
\Psi_0(x)=x
\end{array}
\right.
\end{equation}
and the induced Ricci flow $\bar g(t):=\Psi_t^*g(t)$ on $\mathbb{R}^n\times [0,+\infty)$. Since $\bar g(t)$ is diffeomorphic to $g(t)$, we automatically have $|\Rm(\bar g(t))|\leq t^{-1}$ on $\mathbb{R}^n\times (0,1]$. Moreover since $\Psi_0=\mathrm{Id}$, we have $\bar g(0)=g_0$ on $\mathbb{R}^n$. 

Since the scalar curvature under Ricci flow satisfies $\mathrm{scal}_{\bar g(t)}\geq -C_n t^{-1}$ and 
\begin{equation}
\left(\partial_t -\Delta_{\bar g(t)} \right) \mathrm{scal}_{\bar g(t)} =2|\Ric(\bar g(t))|^2\geq 0. 
\end{equation}

By applying \cite[Theorem 1.1]{LeeTam2022} to $\mathrm{scal}_{\bar g(t)}-\kappa$ on $B_{g_0}(x,2)\Subset B_{g_0}(x_0,4)$, this shows that for any $\ell\in \mathbb{N}$, there exists $T_\ell(n)\in (0,1)$ such that 
\begin{equation}
\mathrm{scal}_{\bar g(t)}(x)-\kappa \geq -t^\ell
\end{equation} 
on $B_{g_0}(x_0,2)\times [0,T_\ell]$.  Since $\Psi_t$ is a isometry, this also implies 
\begin{equation}
\mathrm{scal}_{ g(t)}(x)-\kappa \geq -t^\ell
\end{equation}
on $\Psi_t\left(B_{g_0}(x_0,2)\right)$ for $t\in [0,T_\ell]$. 

It remains to show that $x_0$ contains inside the image.
\begin{claim}\label{claim:contain}
For sufficiently small $T_\ell$, we have $$x_0\in \Psi_t\left( B_{g_0}(x_0,2)\right)$$ for all $t\in [0,T_\ell]$.
\end{claim}
\begin{proof}[Proof of Claim~\ref{claim:contain}]
By \cite[Corollary 3.3]{SimonTopping2022}, we might assume $T_\ell$ is small enough so that for all $t\in [0,T_\ell]$,
\begin{equation}
\begin{split}
B_{g(t)}(\Psi_t(x_0),1)=\Psi_t \left(B_{\bar g(t)}(x_0,1)\right)\subseteq \Psi_t\left(B_{g_0}(x_0,2)\right).
\end{split}
\end{equation}

Together with the fact that $|\partial_t \Psi_t(x_0)|\leq C_n t^{-1/2}$, we have 
\begin{equation}
\begin{split}
d_{g(t)}(\Psi_t(x_0),x_0)&\leq 2 d_{g_{euc}}(\Psi_t(x_0),x_0)\\
&\leq C_n\sqrt{t}<1
\end{split}
\end{equation}
provided that we shrink $T_\ell$ further. This shows that $x_0\in \Psi_t\left(B_{g_0}(x_0,2)\right)$ for $t\in [0,T_\ell]$ and thus completes the proof.\end{proof}

\end{proof}

The next Lemma is a pseudolocality property of Ricci flow, showing that the curvature will remain under control locally along the Ricci-DeTurck flow. 
\begin{lma}\label{lma:pseudo}
Suppose $g(t)$ is a Ricci-DeTurck flow with respect to $g_{euc}$ on $\mathbb{R}^n\times [0,1]$ such that $g(0)=g_0$ and 
\begin{enumerate}
\item[(i)] $\displaystyle 2^{-1} g_{euc}\leq g(t)\leq 2 g_{euc}$; 
\item[(ii)] $|\nabla^{euc,2}g(t)|+|\nabla^{euc} g(t)|^2\leq t^{-1}$
\end{enumerate}
on $\mathbb{R}^n\times (0,1]$. Suppose $x_0\in \mathbb{R}^n$ and $\ell\in \mathbb{N}$ are such that 
$$\sup_{B_{g_0}(x_0,4)}|\nabla^{g_0,k}\Rm(g_0)|\leq L_k$$ for all $0\leq k\leq \ell$, then there exists $\Lambda_k(n,k,L_k)$ such that  
$$|\nabla^{g(t),k}\Rm(g(x_0,t))|\leq \Lambda_k$$ for all $0\leq k\leq \ell$ and $t\in [0,1]$.
\begin{equation}
\end{equation}
\end{lma}
\begin{proof}
As in the proof of Lemma~\ref{lma:scalar-low-bdd}, we consider the Ricci flow $\bar g(t):=\Psi_t^*g(t)$ on $\mathbb{R}^n\times [0,1]$. It follows from \cite[Corollary 3.2]{Chen2009} and the modified Shi’s higher order estimates \cite[Theorem 14.16]{ChowBook} that 
\begin{equation}
\sup_{B_{g_0}(x_0,2)} |\nabla^{\bar g(t),k}\Rm(\bar g(t))|\leq \Lambda_k
\end{equation}
for $t\in [0,1]$. The conclusion follows from Claim~\ref{claim:contain} using the fact that $\bar g(t)=\Psi_t^*g(t)$.
\end{proof}

\section{quantification under weak convergence}\label{sec:proof}

In this section, we prove the quantitative scalar curvature estimates. We first construct a local smoothing of the metric $g_0$ using the Ricci–DeTurck flow.

\begin{lma}\label{lma:lift}
For any $\e>0$, there exists $\delta(n,\e)>0$ such that the following holds. Suppose $(M^n,g_0)$ is a smooth manifold and $x_0\in M,r>0$ such that 
\begin{enumerate}
\item[(i)] $B_{g_0}(x_0,r)\Subset M$;
\item[(ii)] $|\Rm(g_0)|\leq r^{-2}$ on $B_{g_0}(x_0,r)$.
\end{enumerate}
Then there exists a local diffeomorphism $\Phi:B_{euc}(\delta r)\to B_{g_0}(x_0,\delta r)$ such that $\Phi\left( B_{euc}(\delta r)\right)=B_{g_0}(x_0,\delta r)$, $\Psi(0^n)=x_0$ and
\begin{equation}
(1-\e) g_{euc}\leq \Phi^* g_0\leq (1+\e) g_{euc},\;\;\text{on} \;\; B_{euc}(\delta r).
\end{equation}
If in addition $\mathrm{Vol}_{g_0}\left(B_{g_0}(x_0,r)\right)\geq v_0r^n$ for some $v_0>0$, then $\Phi$ can be chosen to be diffeomorphism if $\delta$ is chosen to be small depending also on $v_0$.
\end{lma}
\begin{proof}
This follows comparison Theorem that the conjugate radius is bounded from below by $C_n^{-1} r$. Hence, we might take $\Phi$ to be the exponential map $\Phi:=\exp_{x_0}$. The almost Euclidean in metric sense can be achieved by shrinking the radius, see \cite[Theorem 4.10]{Hamilton1995} for example. If in addition, the volume is bounded from below, it follows from the curvature upper bound and the injectivity radius estimate from \cite{CheegerGromovTaylor1982} that $\mathrm{inj}_{g_0}(x_0)\geq c(n,v_0)r$. Hence $\Phi$ can be chosen the normal coordinate at $x_0$.
\end{proof}

\medskip

We first prove the quantification under $C^0$ convergence. 
\begin{proof}[Proof of Theorem~\ref{thm:main}]
For convenience, we write $$||\hat g_0-g_0||_{L^\infty(\Omega)}:= \sigma \quad \text{and}\quad \inf_{\Omega} \mathrm{scal}(\hat g_0):=\kappa.$$
We will use $\Lambda_i$ to denote all dimensional constants. We let $\e_0(n)$ be the dimensional constant from Theorem~\ref{thm:stability-RDF} and insist $\sigma\leq \frac13 \e_0$.

We apply Lemma~\ref{lma:lift} with $\e=\frac14 \e_0$, and obtain $1>\delta_0(\e_0,n)>0$ and a local diffeomorphism $\Phi:B_{euc}(\delta_0 r)\to B_{g_0}(x_0,\delta_0 r )$ such that  
\begin{equation}\label{eqn:equiv-lift}
\left\{
\begin{array}{ll}
\left(1-\e_0\right)g_{euc}\leq \Phi^*g_0 \leq \left(1+\e_0\right) g_{euc};\\[2mm]
\left(1-\e_0\right)g_{euc}\leq \Phi^*\hat g_0 \leq \left(1+\e_0\right) g_{euc}
\end{array}
\right.
\end{equation}
on $B_{euc}(\delta_0 r)$. Here we assume $\sigma \leq \frac13\e_0$. Since the conclusion is scaling invariant, it suffices to prove the result for $r=20\delta_0^{-1}$.

\medskip

We choose a smooth function $\phi$ on $\mathbb{R}^n$ such that $\phi\equiv 1$ on $B_{euc}(8)$, vanishes outside $B_{euc}(10)$ and satisfies $|\nabla^{euc}\phi|\leq 10^4$. Define the metric 
\begin{equation}
\left\{
\begin{array}{ll}
 g_0^\sharp:= \phi \Phi^* g_0 + (1-\phi) g_{euc};\\[2mm]
\hat g_0^\sharp:= \phi \Phi^* \hat g_0 + (1-\phi) g_{euc}
\end{array}
\right.
\end{equation}
which is a smooth metric on $\mathbb{R}^n$. From \eqref{eqn:equiv-lift}, both $ g_0^\sharp$ and $\hat g_0^\sharp$ remain $\e_0$-close to $g_{euc}$. Hence, Theorem~\ref{thm:stability-RDF} applies to produce a Ricci-DeTurck flow $\hat g(t),g(t)$ with respect to $g_{euc}$ on $\mathbb{R}^n\times [0,+\infty)$ with $\hat g(0)=\hat g_0^\sharp$, $g(0)=g_0^\sharp$ on $\mathbb{R}^n$ and
\begin{equation}\label{eqn:RDF-esti}
\left\{
\begin{array}{ll}
 (1-C_0\e_0) g_{euc}\leq g(t),\hat g(t)\leq (1+C_0\e_0)g_{euc};\\[1mm]
 |\nabla^{euc} g(t)|^2+ |\nabla^{euc,2}g(t)|\leq t^{-1};\\[1mm]
 |\nabla^{euc} \hat g(t)|^2+ |\nabla^{euc,2}\hat g(t)|\leq t^{-1}
\end{array}
\right.
\end{equation}
on $\mathbb{R}^n\times (0,+\infty)$. Furthermore,   
\begin{equation}\label{eqn:RDF-stabi}
\begin{split}
 \sum_{k=0}^2 t^{k/2} ||\nabla^{euc,k}\left(\hat g(t)-g(t)\right)||_{L^\infty,g_{euc}}
&\leq C_0 ||\hat g_0^\sharp-g_0^\sharp||_{L^\infty,g_{euc}}\leq  C_0\sigma.
\end{split}
\end{equation}

In particular, \eqref{eqn:RDF-stabi} implies that for all $(x,t)\in \mathbb{R}^n\times (0,1]$,
\begin{equation}\label{eqn:scal-comp}
\mathrm{scal}_{\hat g(t)}(x)\leq  \tr_{\hat g(t)}\Ric(g(t))|_{x}+ \Lambda_1\sigma t^{-1}.
\end{equation}

We consider the Ricci-DeTurck ODE with respect to $g(t)$:
\begin{equation}
\left\{
\begin{array}{ll}
\partial_t \Psi_t(x):=-W(\Psi_t(x),t);\\
\Psi_0(x)=x.
\end{array}
\right.
\end{equation}
and now estimate the left hand side from below, when it is evaluated at $x_t:=\Psi_t(0^n)$.
\begin{claim}\label{claim:scal-ref-ref}
For any $\ell\in \mathbb{N}$, there exists $T_\ell(n)\in (0,1)$ such that $$\mathrm{scal}_{\hat g(t)}(x_t)\geq \kappa -t^\ell$$ for all $t\in [0,T_\ell]$.
\end{claim}
\begin{proof}[Proof of Claim~\ref{claim:scal-ref-ref}]
Recall that $\hat g(t)$ is a Ricci-DeTurck flow with $\hat g(0)=\Phi^* \hat g_0$ on $B_{euc}(8)=\Phi^*\left( B_{g_0}(x_0,8)\right)$. By \eqref{eqn:RDF-esti}, we might assume  $B_{\hat g(0)}(0^n,6)\Subset B_{euc}(8)$ by shrinking $\e_0$ if necessary. For $x\in B_{\hat g(0)}(0^n,2)$ so that $B_{\hat g(0)}(x,4)\subseteq B_{euc}(8)$, we apply  Lemma~\ref{lma:scalar-low-bdd} using \eqref{eqn:RDF-esti} to show that 
$$\mathrm{scal}_{\hat g(t)}(x)\geq \kappa -t^\ell$$ for all $(x,t)\in  B_{\hat g(0)}(0^n,2)\times  [0,T_\ell]$.  Since $\partial_t \Psi_t=|W|\leq t^{-1/2}$, we have $x_t\in B_{euc}(0^n,2)$ if $T_\ell$ is small enough. Result follows.
\end{proof}

\medskip

We now control the right hand side, at the origin $x_t\in \mathbb{R}^n$.
\begin{claim}\label{claim:upper}
There exists $\Lambda_2>0$ such that for all $t\in [0,1]$, we have 
\begin{equation}
\tr_{\hat g(t)}\Ric(g(t))|_{x=x_t}\leq \mathrm{scal}_{g_0}(x_0)+\Lambda_2(t+\sigma).
\end{equation}
\end{claim}
\begin{proof}[Proof of Claim~\ref{claim:upper}]
Similar to the proof of Claim~\ref{claim:scal-ref-ref}, using also  the fact that $\Phi: B_{g(0)}(0^n,4)\to B_{g_0}(x_0,4)$ is a local isometry, it follows from curvature assumptions and Lemma~\ref{lma:pseudo} that 
\begin{equation}\label{eqn:psei}
\sum_{k=0}^2 |\nabla^{g(t),k} \Rm(g(x,t))|\leq \Lambda_3
\end{equation}
for all $(x,t)\in B_{euc}(0^n,2)\times [0,1]$. In particular \eqref{eqn:psei} holds at $x_t$ since $x_t\in B_{euc}(0^n,2)$ for $t$ sufficiently small (depending only on dimension). Using the fact that 
\begin{equation}
\begin{split}
\partial_t \mathrm{scal}_{g(t)}(x_t)=\Delta_{g(t)}\mathrm{scal}_{g(t)}|_{x=x_t}+2|\Ric(g(t))|^2(x_t)\leq \Lambda_4.
\end{split}
\end{equation}

This together with \eqref{eqn:RDF-stabi} and \eqref{eqn:psei} shows that at $x_t$, we have
\begin{equation}
\begin{split}
\hat g^{ij}R_{ij} &=(\hat g^{ij}-g^{ij})R_{ij}+\mathrm{scal}_{g(t)}\\
&\leq \mathrm{scal}_{g_0}(x_0)+\Lambda_5\sigma +\Lambda_5 t.
\end{split}
\end{equation}
Here we have used $\mathrm{scal}_{g_0}(x_0)=\mathrm{scal}_{g(0)}(0^n)$, since $\Phi$ is a local isometry with $\Phi(0^n)=x_0$.
\end{proof}

Combining \eqref{eqn:scal-comp}, Claim~\ref{claim:scal-ref-ref} and Claim~\ref{claim:upper}, we conclude that there is $T_1(n)>0$ such that for all $t\in (0,T_1]$ (by fixing $\ell=1$),
\begin{equation}\label{eqn:sca-approx}
\kappa \leq \mathrm{scal}_{g_0}(x_0)+\Lambda_6 (t +\sigma t^{-1}).
\end{equation}

Since $\sigma\leq \e_0$, we might assume $\sigma^{1/2}\leq T_1$. Result follows by choosing $t=\sigma^{1/2}$.
\end{proof}

\medskip

We next consider the case when $||g_0-\hat g_0||_{L^\infty(\Omega,g_0)}\leq \e_n$ and $||g_0-\hat g_0||_{L^p(\Omega,g_0)}:=\sigma<<\e_n$ for $p\in [1,+\infty)$. 
\begin{proof}[Proof of Theorem~\ref{thm:Lp-quant}]
This follows from modification of proof of Theorem~\ref{thm:main}. We only give a sketch. We use $L_i$ to denote constants depending only on $n,v_0$.  We also denote $\sigma:=||g_0-\hat g_0||_{L^p(\Omega,g_0)}$.

As in the proof of Theorem~\ref{thm:main}, we consider the local property of the Ricci-DeTurck flow $g(t)$ and $\hat g(t)$, restricted on the lifted ball $B_{euc}(6)$, where now $\Phi$ from Lemma~\ref{lma:lift} is a diffeomorphism instead of a local diffeomorphism. In particular their initial data satisfies 
\begin{equation}
\begin{split}
||g_0^\sharp-\hat g_0^\sharp||_{L^p(\mathbb{R}^n)}&\leq ||\Phi^*g_0-\hat \Phi^* g_0||_{L^p(B_{euc}(10),g_{euc})}\\
&\leq L_1 ||\Phi^*g_0-\Phi^*\hat  g_0||_{L^p(B_{euc}(10),\Phi^* g_{0})}\\
&=L_1 ||g_0-\hat g_0||_{L^p(B_{g_0}(x_0,10),g_{0})}\leq L_1\sigma.
\end{split}
\end{equation}

Hence Proposition~\ref{prop:improve-L-infty-Lp} implies an improved estimate: 
\begin{equation}
||g(t)-\hat g(t)||_{L^\infty(g_{euc})}\leq L_2\sigma t^{-\frac{n}{2p}} 
\end{equation}
for all $t\in (0,1]$. For each $t\in (0,1]$ by  considering the time interval $[t/2,t]$, the standard higher order estimates (for example see the proof of \cite[Lemma 2.6]{Lee2025}) implies that 
\begin{equation}
t^{k/2}||\nabla^{euc,k}\left(g(t)-\hat g(t)\right)||_{L^\infty(g_{euc})}\leq L_2\sigma t^{-\frac{n}{2p}}.
\end{equation}

Now we are in position to carry out the same analysis as in the proof of Theorem~\ref{thm:main}, except now \eqref{eqn:scal-comp} is replaced by 
\begin{equation}\label{eqn:scal-comp-collaps}
\mathrm{scal}_{\hat g(t)}(x)\leq  \tr_{\hat g(t)}\Ric(g(t))|_{x}+ L_3\sigma t^{-1-\frac{n}{2p}}.
\end{equation}

By combining \eqref{eqn:scal-comp-collaps} with Claim~\ref{claim:scal-ref-ref} and Claim~\ref{claim:upper}, we  conclude that
\begin{equation}
\kappa \leq \mathrm{scal}_{g_0}(x_0)+L_4 \left(t+ t^{-1} \sigma t^{-\frac{n}{2p}}\right),
\end{equation}
in contrast with \eqref{eqn:sca-approx}. Result follows by choosing $t=\sigma^\frac1{2+\frac{n}{2p}}$.
\end{proof}

\begin{rem}
It is clear from the proof of Theorem~\ref{thm:main} and Theorem~\ref{thm:Lp-quant}, that quantification still holds as long as some quantitative continuity of curvature is assumed. It is however unclear whether the resulting estimate is sharp or not.
\end{rem}

\appendix

\section{Example of Mazurowski-Yao}\label{sec:appe}

In this section, we show that the example constructed by Mazurowski-Yao \cite[Section 2]{MazurowskiYao2026} has already showed the sharpness of the exponent in Theorem~\ref{thm:Lp-quant} for $1\leq p<+\infty$. The construction in \cite{MazurowskiYao2026} is based on conformal deformation of Euclidean metric on unit ball $B_{euc}(1)$. 

For $r^4:=\e> 0$, take a cutoff function $\eta$ on $B_{euc}(1)$  such that $\eta\equiv 1$ on $B_{euc}(r/2)$, vanishes outside $B_{euc}(r)$ and satisfies 
\begin{equation}
|\nabla^{euc} \eta|^2+|\nabla^{euc,2}\eta|\leq r^{-2}
\end{equation}
for some dimensional constant $C>0$. Let also 
\begin{equation}
v:=\frac{r^2-|x|^2}{2n}.
\end{equation}

Mazurowski-Yao consider the metric $g_0:=\phi_0^\frac4{n-2} g_{euc}$ and $g_\e:=\phi_\e^\frac{4}{n-2}g_{euc}$  on $B_{euc}(1)$, where 
\begin{equation}
\phi_\e:=1-\frac{|x|^2}{4(2+n)}+a \e^{1/2} \eta v 
\end{equation}
for $\e\geq 0$ and $a>0$. In particular, direct computation shows that 
\begin{equation}
\mathrm{scal}_{g_0}(0^n)=0,\;\;\text{and}\;\; \sum_{k=0}^2|\nabla^{g_0,k}\Rm(g_0)|\leq C_n.
\end{equation}
Furthermore, it was shown by Mazurowski-Yao \cite{MazurowskiYao2026} that provided that $a$ is small enough (depending on dimensional constant), we have $\inf_{B_{euc}(1)}\mathrm{scal}_{g_\e}\geq c_n\e^{1/2}$ while
\begin{equation}
\begin{split}
||g_\e-g_0||_{L^p}&\leq C_n || \phi_\e-\phi_0||_{L^p(B_{euc}(r))}\\
&\leq C_n \e^{\frac12} r^{2+\frac{n}p}=C_n \e^{1+\frac{n}{4p}}.
\end{split}
\end{equation}
This shows that the optimal exponent is at most $(2+\frac{n}{2p})^{-1}$.

\end{document}